\documentclass[12pt,a4paper]{article}
\usepackage{mathrsfs}
\usepackage{indentfirst}
\usepackage{amsmath,amssymb,amsfonts,amsthm,graphics}
\usepackage{makeidx}
\usepackage{color}
\include{PDF}
\newtheorem{theorem}{Theorem}[section]
\newtheorem{lemma}[theorem]{Lemma}

\newtheorem{corollary}[theorem]{Corollary}

\title{\bf The $(k,\ell)$-rainbow index for complete bipartite and
multipartite graphs\footnote{Supported by NSFC and the ``973'' program.}}
\author{\small Qingqiong Cai, Xueliang Li, Jiangli Song\\
{\small Center for Combinatorics and LPMC-TJKLC}\\ {\small Nankai
University}\\ {\small Tianjin 300071, China}\\ {\small Email:
cqqnjnu620@163.com, lxl@nankai.edu.cn,
songjiangli@mail.nankai.edu.cn}}
\date{}

\begin{document}
\maketitle

\begin{abstract}
A tree in an edge-colored graph $G$ is said to be a rainbow tree if
no two edges on the tree share the same color. Given two positive
integers $k$, $\ell$ with $k\geq 3$, the \emph{$(k,\ell)$-rainbow
index} $rx_{k,\ell}(G)$ of $G$ is the minimum number of colors
needed in an edge-coloring of $G$ such that for any set $S$ of $k$
vertices of $G$, there exist $\ell$ internally disjoint rainbow
trees connecting $S$. This concept was introduced by Chartrand et
al., and there have been very few results about it. In this paper,
we investigate the $(k,\ell)$-rainbow index for complete bipartite
graphs and complete multipartite graphs. Some asymptotic values of
their $(k,\ell)$-rainbow index are obtained.

\end{abstract}



\section {\large Introduction}

All graphs in this paper are undirected, finite and simple. We
follow \cite{Bondy} for graph theoretical notation and terminology
not defined here. Let $G$ be a nontrivial connected graph with an
\emph{edge-coloring} $c: E(G)\rightarrow\{1, 2,\cdots, t\}, t \in
\mathbb{N}$, where adjacent edges may be colored the same. A path of
$G$ is said to be a \emph{rainbow path} if no two edges on the path
have the same color. An edge-colored graph $G$ is called
\emph{rainbow connected} if for every pair of distinct vertices of
$G$ there exists a rainbow path connecting them. The \emph{rainbow
connection number} of $G$, denoted by $rc(G)$, is defined as the
minimum number of colors that are needed in order to make $G$
rainbow connected. The \emph{rainbow $k$-connectivity} of $G$,
denoted by $rc_{k}(G)$, is defined as the minimum number of colors
in an edge-coloring of $G$ such that every two distinct vertices of
$G$ are connected by $k$ internally disjoint rainbow paths. These
concepts were introduced by Chartrand et al. in \cite{ChartrandGP,
ChartrandGL}. Recently, a lot of results on the rainbow connections
have been published. We refer the reader to \cite{LiSun3, LiSun} for
details.

Similarly, a tree $T$ in $G$ is called a \emph{rainbow tree} if no
two edges of $T$ have the same color. For $S\subseteq V(G)$, a
\emph{rainbow $S$-tree} is a rainbow tree connecting the vertices of
$S$. Suppose $\{T_{1},T_{2},\cdots, T_{\ell}\}$ is a set of rainbow
$S$-trees. They are called \emph{internally disjoint} if
$E(T_{i})\cap E(T_{j})=\emptyset$ and $V(T_{i})\bigcap V(T_{j})=S$
for every pair of distinct integers $i,j$ with $1\leq i,j\leq \ell$
(note that these trees are vertex-disjoint in $G\setminus S$). Given
two positive integers $k$, $\ell$ with $k\geq 2$, the
\emph{$(k,\ell)$-rainbow index} $rx_{k,\ell}(G)$ of $G$ is the
minimum number of colors needed in an edge-coloring of $G$ such that
for any set $S$ of $k$ vertices of $G$, there exist $\ell$
internally disjoint rainbow $S$-trees. In particular, for $\ell=1$,
we often write $rx_{k}(G)$ rather than $rx_{k,1}(G)$ and call it the
\emph{$k$-rainbow index}. It is easy to see that
$rx_{2,\ell}(G)=rc_{\ell}(G)$. So the $(k,\ell)$-rainbow index can
be viewed as a generalization of the rainbow connectivity. In the
sequel, we always assume $k\geq 3$.

The concept of $(k,\ell)$-rainbow index was also introduced by
Chartrand et al.; see \cite{Chartrand}. They determined the
$k$-rainbow index of all unicyclic graphs and the $(3,\ell)$-rainbow
index of complete graphs for $\ell=1,2$. In \cite{Cai}, we
investigated the $(k,\ell)$-rainbow index of complete graphs. We
proved that for every pair of positive integers $k,\ell$ with $k\geq
3$, there exists a positive integer $N=N(k,\ell)$ such that
$rx_{k,\ell}(K_{n})=k$ for every integer $n\geq N$, which settled
down the two conjectures in \cite{Chartrand}.

In this paper, we apply the probabilistic method \cite{Alon} to
study a similar question for complete bipartite graphs and complete
multipartite graphs. It is shown that for $k\geq 4$,
$rx_{k,\ell}(K_{n,n})=k+1$  when $n$ is sufficiently large; whereas
for $k=3$, $rx_{k,\ell}(K_{n,n})=3$ for $\ell =1,2$ and
$rx_{k,\ell}(K_{n,n})=4$ for $\ell \geq3$ when $n$ is sufficiently
large. Moreover, we prove that when $n$ is sufficiently large,
$rx_{k,\ell}(K_{r\times n})=k$ for $k<r$; $rx_{k,\ell}(K_{r\times
n})=k+1$ for $k\geq r$ and
$\ell>\frac{{r\choose2}\lceil\frac{k}{r}\rceil^{2}}{\lfloor\frac{k}{r}\rfloor}$;
$rx_{k,\ell}(K_{r\times n})=k$ or $k+1$ for $k\geq r$ and
$\ell\leq\frac{{r\choose2}\lceil\frac{k}{r}\rceil^{2}}{\lfloor\frac{k}{r}\rfloor}$.
At the end of this paper, we totally determine the
$(3,\ell)$-rainbow index of $K_{r\times n}$ for sufficiently large
$n$. Note that all these results can be expanded to more general
complete bipartite graphs $K_{m,n}$ with $m=O(n^{\alpha})$ and
complete multipartite graphs $K_{n_{1},n_{2},\cdots,n_{r}}$ with
$n_{1}\leq n_{2}\leq\cdots\leq n_{r}$ and $n_{r}=O(n_{1}^{\alpha})$,
where $\alpha\in \mathbb{R}$ and $\alpha\geq 1$.

\section{The $(k,\ell)$-rainbow index for complete bipartite\\ graphs}

This section is devoted to the $(k,\ell)$-rainbow index for complete
bipartite graphs. We start with a lemma about the regular complete
bipartite graphs $K_{n,n}$.
\begin{lemma}\label{lem1}
For every pair of positive integers $k,\ell$ with $k\geq3$, there
exists a positive integer $N=N(k,\ell)$, such that
$rx_{k,\ell}(K_{n,n})\leq k+1$ for every integer $n \geq N$.
\end{lemma}
\begin{proof}
Let $C=\{1,2,\cdots,k+1\}$ be a set of $k+1$ different colors. We
color the edges of $K_{n,n}$ with the colors from $C$ randomly and
independently. For $S\subseteq V(K_{n,n})$ with $|S|=k$, define
$A(S)$ as the event that there exist at least $\ell$ internally
disjoint rainbow $S$-trees. If Pr[ $\bigcap\limits_{S}A(S)$ ]$>0$,
then there exists a required $(k+1)$-edge-coloring, which implies
$rx_{k,\ell}(K_{n,n})\leq k+1$.

Assume that $K_{n,n}=G[U,V]$, where
$U={\{u_{1},u_{2},\cdots,u_{n}\}}$ and
$V={\{v_{1},v_{2},\cdots,v_{n}\}}$. We distinguish the following
three cases.

\emph{Case 1: $S\subseteq U$.}

Without loss of generality, we suppose $S=\{u_{1}, u_{2}, \cdots,
u_{k}\}$. For any vertex $v_{i}\in V$, let $T(v_{i})$ denote the
star with $v_{i}$ as its center and $E(T(v_{i}))=\{u_{1}v_{i},
u_{2}v_{i},\cdots, u_{k}v_{i}\}$. Clearly, $T(v_{i})$ is an
$S$-tree. Moreover, for $v_{i}, v_{j} \in V$ and $v_{i}\neq v_{j}$,
$T(v_{i})$ and $T(v_{j})$ are two internally disjoint $S$-trees. Let
$\mathcal{T}_{1}=\{T(v_{i})|v_{i} \in V, 1\leq i\leq n\}$. Then
$\mathcal{T}_{1}$ is a set of $n$ internally disjoint $S$-trees. It
is easy to see that $p_{1}$:= Pr[$T\in \mathcal{T}_{1}$ is a rainbow
$S$-tree]$=(k+1)!/(k+1)^{k}$. Define $X_{1}$ as the number of
rainbow S-trees in $\mathcal{T}_{1}$. Then we have
\begin{center}
$Pr[\overline{A(S)}]\leq Pr[X_{1}\leq\ell-1]\leq {n \choose \ell-1}(1-p_{1})^{n-\ell+1}\leq n^{\ell-1}(1-p_{1})^{n-\ell+1}$
\end{center}

\emph{Case 2: $S\subseteq V$.}

Similar to $Case\ 1$, we get $Pr[\overline{A(S)}]\leq
n^{\ell-1}(1-p_{1})^{n-\ell+1}.$

\emph{Case 3: $S\cap U \neq\emptyset$, $S\cap V \neq\emptyset$.}

Assume that $U^{\prime}=S\cap
U=\{u_{x_{1}},u_{x_{2}},\cdots,u_{x_{a}}\}$ and $V^{\prime}=S\cap
V=\{v_{y_{1}},v_{y_{2}},\cdots,v_{y_{b}}\}$, where $x_{i},y_{i}\in
\{1,2,\cdots, n\}$, $a\geq1$, $b\geq1$ and $a+b=k$. For every pair
$\{u_{i},v_{i}\}$ of vertices with $u_{i}\in U\setminus U^{\prime}$
and $v_{i}\in V\setminus V^{\prime}$, let $T(u_{i}v_{i})$ denote the
S-tree, where $V(T(u_{i}v_{i}))=S \cup \{u_{i},v_{i}\}$ and
$E(T(u_{i}v_{i}))=\{u_{i}v_{i},u_{i}v_{y_{1}},u_{i}v_{y_{2}},
\cdots, u_{i}v_{y_{b}}, v_{i}u_{x_{1}},v_{i}u_{x_{2}}, \cdots,
v_{i}u_{x_{a}}\}$. Clearly, for $i\neq j$, $T(u_{i}v_{i})$ and
$T(u_{j}v_{j})$ are two internally disjoint $S$-trees. Let
$\mathcal{T}_{2}=\{T(u_{i}v_{i})|u_{i}\in U\setminus U^{\prime},
v_{i}\in V\setminus V^{\prime}\}$. Then $\mathcal{T}_{2}$ is a set
of $n-d$ $(max\{a,b\}\leq d \leq k)$ internally disjoint $S$-trees.
It is easy to see that $p_{2}:=$Pr[T$\in \mathcal{T}_{2}$ is a
rainbow tree]=$(k+1)!/(k+1)^{k+1}$. Define $X_{2}$ as the number of
rainbow S-trees in $\mathcal{T}_{2}$. Then we have
\begin{center}
$Pr[\overline{A(S)}]\leq Pr[X_{2}\leq\ell-1]\leq {n-d \choose
\ell-1} (1-p_{2})^{n-d-\ell+1}\leq
n^{\ell-1}(1-p_{2})^{n-k-\ell+1}$.
\end{center}

Comparing the above three cases, we get $Pr[\overline{A(S)}]\leq
n^{\ell-1}(1-p_{2})^{n-k-\ell+1}$ for every set $S$ of $k$ vertices
in $K_{n,n}$. From the union bound, it follows that
\begin{eqnarray*}
  Pr[\ \bigcap\limits_{S}A(S)\ ] &=& 1-Pr[\ \bigcup\limits_{S}\overline{A(S)}\ ] \\
   &\geq & 1- \sum\limits_{S}Pr[\overline{A(S)}] \\
   &\geq & 1- {2n \choose k}n^{\ell-1}(1-p_{2})^{n-k-\ell+1} \\
   &\geq & 1-2^{k}n^{k+\ell-1}(1-p_{2})^{n-k-\ell+1}
\end{eqnarray*}
Since $\lim\limits_{n\rightarrow
\infty}1-2^{k}n^{k+\ell-1}(1-p_{2})^{n-k-\ell+1}=1$, there exists a
positive integer $N=N(k,\ell)$ such that $Pr[\
\bigcap\limits_{S}A(S)\ ]>0$ for all integers $n\geq N$, and thus
$rx_{k,\ell}(K_{n,n})\leq k+1$.
\end{proof}

We proceed with the definition of bipartite Ramsey number, which is
used to derive a lower bound for $rx_{k,\ell}(K_{n,n})$. Classical
Ramsey number involves coloring the edges of a complete graph to
avoid monochromatic cliques, while bipartite Ramsey number involves
coloring the edges of a complete bipartite graph to avoid
monochromatic bicliques. In \cite{Hattingh}, Hattingh and Henning
defined the {\it bipartite Ramsey number} $b(t,s)$ as the least
positive integer $n$ such that in any red-blue coloring of the edges
of $K(n, n)$, there exists a red $K(t,t)$ (that is, a copy of
$K(t,t)$ with all edges red) or a blue $K(s,s)$. More generally, one
may define the bipartite Ramsey number $b(t_{1},t_{2}, \cdots
,t_{k})$ as the least positive integer $n$ such that any coloring of
the edges of $K(n, n)$ with $k$ colors will result in a copy of
$K(t_{i},t_{i})$ in the $i$th color for some $i$. If $t_{i} = t$ for
all $i$, we denote the number by $b_{k}(t)$. The existence of all
such numbers follows from a result first proved by Erd\H{o}s and
Rado \cite{erdos}, and later by Chv$\acute{a}$tal \cite{Chvatal}.
The known bounds for $b(t,t)$ are
$(1+o(1))\frac{t}{e}(\sqrt{2})^{t+1}\leq b(t,t) \leq
(1+o(1))2^{t+1}\log t$, where the $\log$ is taken to the base 2. The
proof of the lower bound \cite{Hattingh} is an application of the
\emph{Lov$\acute{a}$sz Local Lemma}, while the upper bound
\cite{Conlon} is proved upon the observation that, in order for a
two-colored bipartite graph $K_{m,n}$ to necessarily contain a
monochromatic $K_{k,k}$, it is only necessary that one of $m$ and
$n$ be very large. With similar arguments, we can obtain the bounds
for $b_{k}(t)$ as $(1+o(1))\frac{t}{e}(\sqrt{k})^{t+1}\leq b_{k}(t)
\leq (1+o(1))k^{t+1}\log_{k}t$.

\begin{lemma}\label{lem2}
For every pair of positive integers $k,\ell$ with $k\geq4$, if $n
\geq b_{k}(k)$, then $rx_{k,\ell}(K_{n,n})\geq k+1$.
\end{lemma}
\begin{proof}
By contradiction. Suppose $\textbf{\textit{c}}$ is a
$k$-edge-coloring of $K_{n,n}$ such that for any set $S$ of $k$
vertices in $K_{n,n}$, there exist $\ell$ internally disjoint
rainbow $S$-trees. From the definition of bipartite Ramsey number,
we know that if $n \geq b_{k}(k)$, then in this $k$-edge-coloring
$\textbf{\textit{c}}$, one will find a monochromatic subgraph
$K_{k,k}$. Let $K_{n,n}=G[U,V]$ and $U^{\prime}$, $V^{\prime}$ be
the bipartition of the monochromatic $K_{k,k}$, where
$U^{\prime}\subset U$, $V^{\prime}\subset V$ and
$|U^{\prime}|=|V^{\prime}|=k$. Now take $S$ as follows: two vertices
are from $V^{\prime}$ and the other $k-2 \ (\geq2)$ vertices are
from $U^{\prime}$. Assume that $T$ is one of the $\ell$ internally
disjoint rainbow $S$-trees. Since there are $k$ different colors,
the rainbow tree $T$ contains at most $k+1$ vertices (i.e., at most
one vertex in $V\setminus S$). It is easy to see that $T$ is in the
possession of at least two edges from the subgraph induced by $S$,
which share the same color. It contradicts the fact that $T$ is a
rainbow tree. Thus $rx_{k,\ell}(K_{n,n})\geq k+1$ when $k\geq4$ and
$n\geq b_{k}(k)$.
\end{proof}
From Lemmas \ref{lem1} and \ref{lem2}, we get that if $k\geq4$,
$rx_{k,\ell}(K_{n,n})=k+1$ for  sufficiently large $n$. What remains
to deal with is the $(3,\ell)$-rainbow index of $K_{n,n}$.

\begin{lemma}\label{lem3}
For every integer $\ell\geq 1$, there exists a positive integer $N=N(\ell)$ such that
\begin{displaymath}
rx_{3,\ell}(K_{n,n}) = \left\{ \begin{array}{ll}
3 & \textrm{if $\ell=1,2 $},\\
4 & \textrm{if $\ell\geq3$\ \ \ \ }
\end{array} \right.
\end{displaymath}
for every integer $n\geq N$.
\end{lemma}
\begin{proof}
Assume that $K_{n,n}=G[U,V]$ and $|U|=|V|=n$.

\emph{\textbf{Claim $1$:}} If $\ell\geq1$,
$rx_{3,\ell}(K_{n,n})\geq3$; furthermore, if $\ell\geq3$,
$rx_{3,\ell}(K_{n,n})\geq4$.

\indent If $S=\{x,y,z\}\subseteq U$, then the size of $S$-trees is
at least three, which implies that $rx_{3,\ell}(K_{n,n})\geq3$ for
all integers $\ell\geq1$. If $S=\{x,y,z\}, \{x, y\}\subseteq U, z\in
V$, then the number of internally disjoint $S$-trees of size two or
three is no more than two. Thus we have $rx_{3,\ell}(K_{n,n})\geq4$
for all integers $\ell\geq3$.

Combing with $Lemma\ \ref{lem1}$, we get that if $\ell \geq3$, there
exists a positive integer $N=N(\ell)$ such that
$rx_{3,\ell}(K_{n,n})=4$ for every integer $n\geq N$.

\emph{\textbf{Claim $2$:}} If $\ell=1,2$, then there exists a
positive integer $N=N(\ell)$ such that $rx_{3,\ell}(K_{n,n})=3$ for
every integer $n\geq N$.

Note that $3\leq rx_{3,1}(K_{n,n})\leq rx_{3,2}(K_{n,n})$. So it
suffices to prove $rx_{3,2}(K_{n,n})\leq 3$. In other words, we need
to find a $3$-edge-coloring $c:E(K_{n,n})\rightarrow \{1,2,3\}$ such
that for any $S\subseteq V(K_{n,n})$ with $|S|=3$, there are at
least two internally disjoint rainbow $S$-trees. We color the edges
of $K_{n,n}$ with colors $1,2,3$ randomly and independently. For
$S=\{x,y,z\}\subseteq V(K_{n,n})$, define $B(S)$ as the event that
there exist at least $two$ internally disjoint rainbow $S$-trees.
Similar to the proof in $Lemma$ \ref{lem1}, we only need to prove
$Pr[\ \overline{B(S)}\ ]=o(n^{-3})$, since then
Pr[$\bigcap\limits_{S}B(S)$ ]$>0$ and $rx_{3,2}(K_{n,n})\leq 3$ for
sufficiently large $n$. We distinguish the following two cases.

\emph{Case 1}: x, y, z are in the same vertex class.

Without loss of generality, assume that $\{x,y,z\} \subseteq U$. For
any vertex $v\in V$, let $T(v)$ denote the star with $v$ as its
center and $E(T(v))=\{xv, yv,zv\}$. Clearly,
$\mathcal{T}_{3}=\{T(v)|v \in V\}$ is a set of $n$ internally
disjoint $S$-trees and Pr[T$\in \mathcal{T}_{3}$ is a rainbow
tree]=$\frac{3\times2\times1}{3\times3\times3}=\frac{2}{9}$. Define
$X_{3}$ as the number of internally disjoint rainbow S-trees in
$\mathcal{T}_{3}$. Then we have
\begin{center}
$Pr[\overline{B(S)}]=Pr[X_{3}\leq1]={n\choose
1}\frac{2}{9}(1-\frac{2}{9})^{n-1}+(1-\frac{2}{9})^{n}=o(n^{-3}).$
\end{center}

\emph{Case 2}: $x,y,z$ are in two vertex classes. Without loss of
generality, assume that $\{x,y\} \subseteq U$ and $z\in V$.

\emph{Subcase 2.1}: The edges $xz,yz$ share the same color. Without
loss of generality, we assume $c(xz)=c(yz)=1$. For any vertex $v\in
V\setminus \{z\}$, if $\{c(xv), c(yv)\}=\{2,3\}$, then $\{xz, xv,
yv\}$ or $\{yz, xv, yv\}$ induces a rainbow $S$-tree. So $Pr[\{xz,
xv, yv\}$ induces a rainbow $S$-tree$]=Pr[\{yz, xv, yv\}$ induces a
rainbow $S$-tree$]=\frac{2}{9}$. If there do not exist two
internally disjoint rainbow $S$-trees, then we can find at most one
vertex $v\in V\setminus \{z\}$ satisfying $\{c(xv),
c(yv)\}=\{2,3\}$. Thus
\begin{eqnarray*}
Pr[\ \overline{B(S)}|c(xz)=c(yz)\ ]&=&{n-1 \choose 1}\cdot\frac{2}{9}\cdot(1-\frac{2}{9})^{n-2}+(1-\frac{2}{9})^{n-1}\\
&=&\frac{2n+5}{9}(\frac{7}{9})^{n-2}.
\end{eqnarray*}

\emph{Subcase 2.2}: The edges $xz,yz$ have distinct colors. Without
loss of generality, we assume $c(xz)=1, c(yz)=2$. For any vertex
$v\in V\setminus \{z\}$, if $\{c(xv), c(yv)\}=\{2,3\}$, then $\{xz,
xv, yv\}$ induces a rainbow $S$-tree, and so $Pr[\{xz, xv, yv\}$
induces a rainbow $S$-tree$]=\frac{2}{9}$. Moreover, if $\{c(xv),
c(yv)\}=\{1,3\}$, then $\{yz, xv, yv\}$ induces a rainbow $S$-tree,
and so $Pr[\{yz, xv, yv\}$ induces a rainbow
$S$-tree$]=\frac{2}{9}$. If there do not exist two internally
disjoint rainbow $S$-trees, then we can not find two vertices
$v,v^{\prime}\in V\setminus \{z\}$ satisfying $\{c(xv),
c(yv)\}=\{2,3\}$ and $\{c(xv^{\prime}), c(yv^{\prime})\}=\{1,3\}$.
Thus
\begin{eqnarray*}
Pr[\ \overline{B(S)}|c(xz)\neq c(yz)\ ]&=&(1-\frac{2}{9}-\frac{2}{9})^{n-1}+2\sum_{i=1}^{n-1}{n-1\choose i}(\frac{2}{9})^{i}(1-\frac{2}{9}-\frac{2}{9})^{n-1-i}\\
&=&2(\frac{7}{9})^{n-1}-(\frac{5}{9})^{n-1}.
\end{eqnarray*}
From the law of total probability, we have
\begin{eqnarray*}
Pr[\ \overline{B(S)}\ ]&=&Pr[\ c(xz)= c(yz)\ ]\cdot Pr[\ \overline{B(S)}|c(xz)= c(yz)\ ]\\
&\ \ &+Pr[\ c(xz)\neq c(yz)\ ]\cdot Pr[\ \overline{B(S)}|c(xz)\neq c(yz)\ ]\\
&=&\frac{1}{3}\cdot\frac{2n+5}{9}(\frac{7}{9})^{n-2}
+\frac{2}{3}\cdot[2(\frac{7}{9})^{n-1}-(\frac{5}{9})^{n-1}]\\
&\leq& 2n(\frac{7}{9})^{n-2}=o(n^{-3}).
\end{eqnarray*}
Thus, there exists a positive integer $N=N(\ell)$ such that
$rx_{3,2}(K_{n,n})\leq3$, and then
$rx_{3,1}(K_{n,n})=rx_{3,2}(K_{n,n})=3$ for all integers $n\geq N$.
The proof is thus complete.
\end{proof}

By Lemmas \ref{lem1}, \ref{lem2} and \ref{lem3}, we come to the
following conclusion.

\begin{theorem}\label{1}
For every pair of positive integers  $k,\ell$ with $k\geq3$, there
exists a positive integer $N=N(k,\ell)$, such that
\begin{displaymath}
rx_{k,\ell}(K_{n,n}) = \left\{ \begin{array}{ll}
3 & \textrm{if $k=3, \ell=1,2$}\\
4 & \textrm{if $k=3, \ell\geq3$}\\
k+1 & \textrm{if $k \geq4 $\ \ \ \ }
\end{array} \right.
\end{displaymath}
 for every integer $n \geq N$.
\end{theorem}

With similar arguments, we can expand this result to more general
complete bipartite graphs $K_{m,n}$, where $m=O(n^{\alpha})$ (i.e.,
$m\leq Cn^{\alpha}$ for some positive constant $C$), $\alpha\in
\mathbb{R}$ and  $\alpha\geq1$.

\begin{corollary}\label{2}
Let $m,n$ be two positive integers with $m=O(n^{\alpha})$,
$\alpha\in \mathbb{R}$ and $\alpha\geq1$. For every pair of positive
integers  $k,\ell$ with $k\geq3$, there exists a positive integer
$N=N(k,\ell)$, such that
\begin{displaymath}
rx_{k,\ell}(K_{m,n}) = \left\{ \begin{array}{ll}
3 & \textrm{if $k=3, \ell=1,2$}\\
4 & \textrm{if $k=3, \ell\geq3$}\\
k+1 & \textrm{if $k \geq4 $\ \ \ \ }
\end{array} \right.
\end{displaymath}
 for every integer $n \geq N$.
\end{corollary}

\section{The $(k,\ell)$-rainbow index for complete multipartite graphs}

In this section, we focus on the $(k,\ell)$-rainbow index for
complete multipartite graphs. Let $K_{r\times n}$ denote the
complete multipartite graph with $r\geq 3$ vertex classes of the
same size $n$. We obtain the following results about
$rx_{k,\ell}(K_{r\times n})$:

\begin{theorem}\label{thm2}
For every triple of positive integers $k,\ell,r$ with $k\geq 3$ and
$r\geq 3$, there exists a positive integer $N=N(k,\ell,r)$ such that
\begin{displaymath}
rx_{k,\ell}(K_{r\times n}) = \left\{ \begin{array}{ll}
k & \textrm{if $k< r$}\\
k \ or \ k+1 & \textrm{if $k\geq r, \ell\leq\frac{{r\choose2}\lceil\frac{k}{r}\rceil^{2}}{\lfloor\frac{k}{r}\rfloor}$} \\
k+1 & \textrm{if $k\geq r, \ell>\frac{{r\choose2}\lceil\frac{k}{r}\rceil^{2}}{\lfloor\frac{k}{r}\rfloor}$}\ \ \ \
\end{array} \right.
\end{displaymath}
for every integer $n\geq N$.
\end{theorem}
\begin{proof}
Assume that $K_{r\times n}=G[V_{1},V_{2},\ldots,V_{r}]$ and
$V_{i}=\{v_{i1},v_{i2},\cdots,v_{in}\}$, $1\leq i\leq r$. For
$S\subseteq V(K_{r\times n})$ with $|S|=k$, define $C(S)$ as the
event that there exist at least $\ell$ internally disjoint rainbow
$S$-trees.

\emph{\textbf{Claim $1$:}} $rx_{k,\ell}(K_{r\times n})\geq k$;
furthermore, if $k\geq r$,
$\ell>\frac{{r\choose2}\lceil\frac{k}{r}\rceil^{2}}{\lfloor\frac{k}{r}\rfloor}
$, then $rx_{k,\ell}(K_{r\times n})\geq k+1$.

\indent Let $S\subseteq V_{1}$. Then the size of $S$-trees is at
least $k$, which implies that $rx_{k,\ell}(K_{r\times n})\geq k$ for
all integers $\ell\geq1$. If $k\geq r$, we can take a set
$S^{\prime}$ of $k$ vertices such that $|S^{\prime}\cap
V_{i}|=\lfloor\frac{k}{r}\rfloor$ or $\lceil\frac{k}{r}\rceil$ for
$1\leq  i\leq r$. Let $H$ denote the subgraph induced by
$S^{\prime}$. We know $|E(H)|\leq
{r\choose2}\lceil\frac{k}{r}\rceil^{2}$. Let
$\mathcal{T}=\{T_{1},T_{2},\cdots, T_{c}\}$ be the set of internally
disjoint $S^{\prime}$-trees with $k$ or $k-1$ edges. Clearly, if
$T\in \mathcal{T}$ is an $S^{\prime}$-tree with $k-1$ edges, then
$|E(T)\cap E(H)|=k-1$; if $T\in \mathcal{T}$ is an $S^{\prime}$-tree
with $k$ edges, then $|E(T)\cap E(H)|\geq
\lfloor\frac{k}{r}\rfloor$. Therefore,
$c\leq\frac{{r\choose2}\lceil\frac{k}{r}\rceil^{2}}{\lfloor\frac{k}{r}\rfloor}$
(note that the bound is sharp for $r=3$, $k=3$). Thus we have
$rx_{k,\ell}(K_{r\times n})\geq k+1$ for $k\geq r$ and
$\ell>\frac{{r\choose2}\lceil\frac{k}{r}\rceil^{2}}{\lfloor\frac{k}{r}\rfloor}$.

\emph{\textbf{Claim $2$:}} If $k<r$, then there exists an integer
$N=N(k,\ell,r)$ such that $rx_{k,\ell}(K_{r\times n})\leq k$ for
every integer $n\geq N$.

We color the edges of $K_{r\times n}$ with $k$ colors randomly and
independently. Since $k<r$, no matter how $S$ is taken, we can
always find a set $V_{i}$ satisfying $V_{i}\cap S=\emptyset$. For
any $v_{ij}\in V_{i}$, let $T(v_{ij})$ denote the star with $v_{ij}$
as its center and $E(T(v_{ij}))=\{v_{ij}s|s\in S\}$. Clearly,
$\mathcal{T}_{4}=\{T(v_{ij})|v_{ij} \in V_{i}\}$ is a set of $n$
internally disjoint $S$-trees. Similar to the proof of Case 1 in
Lemma \ref{lem1}, we get $Pr[\overline{C(S)}]=o(n^{-k})$. So there
exists a positive integer $N=N(k,\ell,r)$ such that $Pr[\
\bigcap\limits_{S}C(S)\ ]>0$, and thus $rx_{k,\ell}(K_{n,n})\leq k$
for all integers $n\geq N$.

It follows from $Claims\ 1$ and $2$ that if $k< r$, there exists a
positive integer $N=N(k,\ell,r)$ such that $rx_{k,\ell}(K_{r\times
n})= k$ for every integer $n\geq N$.

\emph{\textbf{Claim $3$:}} If $k\geq \ell$, then there exists a
positive integer $N=N(k,\ell,r)$ such that $rx_{k,\ell}(K_{r\times
n})\leq k+1$ for every integer $n\geq N$.

Color the edges of $K_{r\times n}$ with $k+1$ colors randomly and
independently. We distinguish the following two cases.

\emph{Case 1}: $S\cap V_{i} =\emptyset$ for some $i$.

We follow the notation $T(v_{ij})$ and $\mathcal{T}_{4}$ in $Claim\
2$. Similarly, $\mathcal{T}_{4}$ is a set of $n$ internally disjoint
$S$-trees and $Pr[\overline{C(S)}]=o(n^{-k})$.

\emph{Case 2}: $S\cap V_{i}\neq\emptyset$ for $1\leq i\leq r$.

We pick up two vertex classes $V_{1}, V_{2}$. Suppose
$V_{1}^{\prime}=S\cap
V_{1}=\{v_{1x_{1}},v_{1x_{2}},\cdots,v_{1x_{a}}\}$ and
$V_{2}^{\prime}=S\cap
V_{2}=\{v_{2y_{1}},v_{2y_{2}},\cdots,v_{2y_{b}}\}$, where
$x_{i},y_{i}\in \{1,2,\cdots, n\}$, $a\geq1$, $b\geq 1$. Note that
$a+b\leq k-r+2$. For every pair $\{v_{1i},v_{2i}\}$ of vertices with
$v_{1i}\in V_{1}\setminus V_{1}^{\prime}$ and $v_{2i}\in
V_{2}\setminus V_{2}^{\prime}$, let $T(v_{1i}v_{2i})$ denote the
S-tree, where $V(T(v_{1i}v_{2i}))=S \cup \{v_{1i},v_{2i}\}$ and
$E(T(v_{1i}v_{2i}))=\{v_{1i}v_{2i}\}\cup\{v_{1i}v|v\in
V_{2}^{\prime}\}\cup\{v_{2i}v|v\in S\backslash V_{2}^{\prime} \}$.
Clearly, for $i\neq j$, $T(v_{1i}v_{2i})$ and $T(v_{1j}v_{2j})$ are
two internally disjoint $S$-trees. Let
$\mathcal{T}_{5}=\{T(v_{1i}v_{2i})|v_{1i}\in V_{1}\setminus
V_{1}^{\prime}, v_{2i}\in V_{2}\setminus V_{2}^{\prime}\}$. Then
$\mathcal{T}_{5}$ is a set of $n-d$ $(max\{a,b\}\leq d \leq a+b \leq
k-r+2)$ internally disjoint $S$-trees. Similar to the proof of Case
3 in Lemma \ref{lem1}, we get $Pr[\overline{C(S)}]=o(n^{-k})$.

Therefore we conclude that there exists a positive integer
$N=N(k,\ell,r)$ such that\\ $Pr[\ \bigcap\limits_{S}C(S)\ ]>0$, and
thus $rx_{k,\ell}(K_{r\times n})\leq k+1$ for all integers $n\geq
N$.

It follows from $Claims$ $1$ and $3$ that if $k\geq r,
\ell>\frac{{r\choose2}\lceil\frac{k}{r}\rceil^{2}}{\lfloor\frac{k}{r}\rfloor}$,
$rx_{k,\ell}(K_{r\times n})=k+1$ and if $k\geq r,
\ell\leq\frac{{r\choose2}\lceil\frac{k}{r}\rceil^{2}}{\lfloor\frac{k}{r}\rfloor}$,
$rx_{k,\ell}(K_{r\times n})=k$ or $k+1$ for sufficiently large $n$.
The proof is thus complete.
\end{proof}

Note that if $k\geq r\geq 3,
\ell\leq\frac{{r\choose2}\lceil\frac{k}{r}\rceil^{2}}{\lfloor\frac{k}{r}\rfloor}$,
we cannot tell $rx_{k,\ell}(K_{r\times n})=k$ or $k+1$ from Theorem
\ref{thm2}. However, the next lemma shows that when $k=r=3$,
$\ell\leq\frac{{r\choose2}\lceil\frac{k}{r}\rceil^{2}}{\lfloor\frac{k}{r}\rfloor}=3$,
$rx_{3,\ell}(K_{n,n,n})=3$ for sufficiently large $n$.
\begin{lemma}\label{lem4}
For $\ell\leq 3$, there exists a positive integer $N=N(\ell)$ such that
$rx_{3,\ell}(K_{n,n,n})=3$ for every integer $n\geq N$.
\end{lemma}
\begin{proof} Assume that $K_{n,n,n}=G[V_{1},V_{2},V_{3}]$ and
$|V_{1}|=|V_{2}|=|V_{3}|=n$. Note that $3\leq
rx_{3,1}(K_{n,n,n})\leq rx_{3,2}(K_{n,n,n})\leq
rx_{3,3}(K_{n,n,n})$. So it suffices to show
$rx_{3,3}(K_{n,n,n})\leq 3$. In other words, we need to find a
$3$-edge-coloring $c: $ $E(K_{n,n,n})\rightarrow \{1,2,3\}$ such
that for any $S=\{u,v,w\}\subseteq V(K_{n,n,n})$, there are at least
three internally disjoint rainbow $S$-trees. We color the edges of
$K_{n,n,n}$ with the colors $1,2,3$ randomly and independently.
Define $D(S)$ as the event that there exist at least $three$
internally disjoint rainbow $S$-trees. We only need to prove
$Pr[\overline{D(S)}]=o(n^{-3})$, since then Pr[
$\bigcap\limits_{S}D(S)$ ]$>0$ and $rx_{3,3}(K_{n,n,n})\leq 3$ for
sufficiently large $n$. We distinguish the following three cases.

\emph{Case 1}: $u,v,w$ are in the same vertex class.

Without loss of generality, assume that $\{u,v,w \}\subseteq V_{1}$.
For any vertex $z\in V_{2}\cup V_{3}$, let $T_{1}(z)$ denote the
star with $z$ as its center and $E(T_{1}(z))=\{zu, zv, zw\}$.
Obviously, $\mathcal{T}_{6}=\{T_{1}(z)|z\in V_{2}\cup V_{3}\}$  is a
set of $2n$ internally disjoint $S$-trees and Pr[T$\in
\mathcal{T}_{6}$ is a rainbow tree]=$\frac{2}{9}$. So,
\begin{center}
$Pr[\ \overline{D(S)}\ ]\leq  {2n \choose 2}(1-\frac{2}{9})^{2n-2}\leq4n^{2}(\frac{7}{9})^{2n-2}=o(n^{-3})$
\end{center}

\emph{Case 2}: $u,v,w$ are in two vertex classes.

Without loss of generality, assume that $\{u,v\} \subseteq V_{1}$
and $w\in V_{2}$. For any vertex $z\in V_{3}$, let $T_{2}(z)$ denote
a star with $z$ as its center and $E(T_{2}(z))=\{zu, zv, zw\}$.
Clearly, $\mathcal{T}_{7}=\{T_{2}(z)|z\in V_{3}\}$ is a set of $n$
internally disjoint $S$-trees and Pr[T$\in \mathcal{T}_{7}$ is a
rainbow tree]=$\frac{2}{9}$. So,
\begin{center}
$Pr[\overline{D(S)}]\leq  {n \choose 2} (1-\frac{2}{9})^{n-2}\leq n^{2}(\frac{7}{9})^{n-2}=o(n^{-3})$
\end{center}

\emph{Case 3}: $u,v,w$ are in three vertex classes.

Assume that $u\in V_{1}$, $v\in V_{2}$ and $w\in V_{3}$. For
$e\in\{uv, vw, wu\}$, define $E_{e}$ as the event that $e$ is not
used to construct a rainbow $S$-tree. Clearly,
$Pr[\overline{D(S)}]\leq Pr[ E_{uv}]+Pr[E_{vw}]+Pr[E_{wu}]$.

\emph{Subcase 3.1}: The edges $uv,vw,wu$ receive distinct colors.

Let $p_{e}=Pr[E_{e}|uv,vw,wu$ receive distinct colors] for $e\in
\{uv, vw, wu\}$. Without loss of generality, we assume $c(uv)=1,\
c(vw)=2,\ c(wu)=3$. If $uv$ is not used to construct a rainbow
$S$-tree, then for any vertex $u'\in V_{1}\setminus \{u\}$, $v'\in
V_{2}\setminus \{v\}$, $\{c(u'v), c(u'w)\}\neq\{2,3\}$ and
$\{c(v'u), c(v'w)\}\neq\{2,3\}$. So
$p_{uv}\leq(1-\frac{2}{9})^{2n-2}$. Similarly
$p_{vw}\leq(1-\frac{2}{9})^{2n-2}$,
$p_{wu}\leq(1-\frac{2}{9})^{2n-2}$. Thus
\begin{center}
$Pr[\overline{D(S)}|uv,vw,wu$ receive distinct colors$]\leq p_{uv}+
p_{vw} +p_{wu}\leq3 (\frac{7}{9})^{2n-2}$
\end{center}

\emph{Subcase 3.2}: the edges $uv,vw,wu$ receive two colors.

Let $p'_{e}=Pr[E_{e}|uv,vw,wu$ receive two colors] for $e\in \{uv,
vw, wu\}$. Without loss of generality, we assume $c(uv)= c(vw)=1,\
c(wu)=2$. If $uv$ is not used to construct a rainbow $S$-tree, then
either

$\bullet$ for any vertex $u'\in V_{1}\setminus \{u\}$, $v'\in
V_{2}\setminus \{v\}$, $\{c(u'v), c(u'w)\}\neq\{2,3\}$ and
$\{c(v'u), c(v'w)\}$ $\neq\{2,3\}$; or

$\bullet$ there exists exactly one vertex $v'\in V_{2}\setminus
\{v\}$ satisfying $\{c(v'u), c(v'w)\}=\{2,3\}$, and at the same
time, for any vertex $u'\in V_{1}\setminus \{u\}$, $w'\in
V_{3}\setminus \{w\}$, $\{c(u'v), c(u'w)\}\neq\{2,3\}$ and
$\{c(w'u), c(w'v)\}\neq\{2,3\}$.

So, $p'_{uv}\leq(\frac{7}{9})^{2n-2}+{n-1 \choose
1}\frac{2}{9}(\frac{7}{9})^{n-2}(\frac{7}{9})^{2n-2}$. Similarly,
$p'_{vw}\leq(\frac{7}{9})^{2n-2}+{n-1 \choose
1}\frac{2}{9}(\frac{7}{9})^{n-2}(\frac{7}{9})^{2n-2}$. Similar to
$p_{wu}$, we have $p'_{wu}\leq(1-\frac{2}{9})^{2n-2}$. Thus
\begin{center}
$Pr[\overline{D(S)}|uv,vw,wu$ receive two colors$] \leq p'_{uv}+
p'_{vw} +p'_{wu}\leq(2n+1) (\frac{7}{9})^{2n-2}$
\end{center}

\emph{Subcase 3.3}: The edges $uv,vw,wu$ receive the same color.

Let $p''_{e}=Pr[E_{e}|uv,vw,wu$ receive the same color] for $e\in
\{uv, vw, wu\}$. Without loss of generality, we assume
$c(uv)=c(vw)=c(wu)=1$. If $uv$ is not used to construct a rainbow
$S$-tree, then one of the three situations below must occur:

$\bullet$ for any vertex $w'\in V_{3}\setminus\{w\}$,
$\{c(w'v),c(w'u)\}\neq \{2,3\}$, and at the same time, there exists
at most one vertex  $u'\in V_{1}\setminus\{u\}$ satisfying
$\{c(u'v),c(u'w)\}= \{2,3\}$ and at most one vertex $v'\in
V_{2}\setminus\{v\}$ satisfying $\{c(v'u),c(v'w)\}= \{2,3\}$.

$\bullet$ there exists exactly one vertex $w'\in
V_{3}\setminus\{w\}$ satisfying $\{c(w'v),c(w'u)\}=\{2,3\}$, and at
the same time, there exists at most one vertex $s\in
(V_{1}\setminus\{u\})\cup (V_{2}\setminus\{v\} )$ satisfying
$\{c(sv),c(sw)\}= \{2,3\}$ or $\{c(su),c(sw)\}= \{2,3\}$.

$\bullet$ there exist at least two vertices $w_{1},w_{2}$ in
$V_{3}\setminus\{w\}$ satisfying $\{c(w_{i}u),c(w_{i}v)\}= \{2,3\}$
for i=1,2, and at the same time, for any vertex $u^{\prime}\in
V_{1}\setminus\{u\}$, $v^{\prime}\in V_{2}\setminus\{v\}$,
$\{c(u^{\prime}v),c(u^{\prime}w)\}\neq\{2,3\}$ and
$\{c(v^{\prime}u),c(v^{\prime}w)\}\neq \{2,3\}$.

So, $p''_{uv}\leq (\frac{7}{9})^{n-1}{n-1 \choose 1}{n-1 \choose
1}(\frac{7}{9})^{2n-4}+{n-1 \choose
1}\frac{2}{9}(\frac{7}{9})^{n-2}{2n-2 \choose
1}(\frac{7}{9})^{2n-3}+[1-(\frac{7}{9})^{n-1}- {n-1 \choose 1}$
$\frac{2}{9}(\frac{7}{9})^{n-2}]\cdot(\frac{7}{9})^{2n-2} $.
Obviously, $p''_{vw}=p''_{wu}=p''_{uv}$. Thus,
\begin{center}
$Pr[\overline{D(S)}|uv,vw,wu$ receive the same color$]\leq p''_{uv}+
p''_{vw} +p''_{wu}\leq3(3n^{2}+1)(\frac{7}{9})^{2n-2}.$
\end{center}
By the law of total probability, we obtain
\begin{eqnarray*}
Pr[\overline{D(S)}]&=&\frac{2}{9}\cdot Pr[\overline{D(S)}|uv,vw,wu\ have\   distinct\ colors]\\
&\ \ &+\frac{2}{3}\cdot Pr[\overline{D(S)}|uv,vw,wu\ receive\   two\ colors]\\
&\ \ &+\frac{1}{9}\cdot Pr[\overline{D(S)}|uv,vw,wu\ share\ the\
same\ color ]\\
&\leq&\frac{2}{9}\cdot 3\cdot (\frac{7}{9})^{2n-2}+\frac{2}{3}(2n+1) (\frac{7}{9})^{2n-2}+\frac{1}{9}\cdot3(3n^{2}+1) (\frac{7}{9})^{2n-2}\\
&=&o(n^{-3}).
\end{eqnarray*}
Therefore, there exists a positive integer $N=N(\ell,r)$ such that
$Pr[\ \bigcap\limits_{S}D(S)\ ]>0$ for all integers $n\geq N$, which
implies $rx_{3,3}(K_{n,n, n})\leq 3$ for $n\geq N$.
\end{proof}

From Theorem \ref{thm2} and Lemma \ref{lem4}, the $(3,\ell)-$rainbow
index of $K_{r\times n}$ is totally determined for sufficiently
large $n$.

\begin{theorem}\label{thm3}
For every pair of positive integers $\ell,r$ with $r\geq3$, there
exists a positive integer $N=N(\ell,r)$ such that
\begin{displaymath}
rx_{3,\ell}(K_{r\times n}) = \left\{ \begin{array}{ll}
3 & \textrm{if $r=3, \ell=1,2,3 $}\\
4 & \textrm{if $r=3, \ell\geq4$} \\
3 & \textrm{if $r\geq4$}\ \ \ \
\end{array} \right.
\end{displaymath}
for every integer $n\geq N$.
\end{theorem}

With similar arguments, we can expand these results to more general
complete multipartite graphs $K_{n_{1},n_{2},\cdots,n_{r}}$ with
$n_{1}\leq n_{2}\leq\cdots\leq n_{r}$ and $n_{r}=O(n_{1}^{\alpha})$,
where $\alpha\in \mathbb{R}$ and $\alpha\geq1$.

\begin{theorem}\label{thm4}
Let $n_{1}\leq n_{2}\leq\cdots\leq n_{r}$ be $r$ positive integers
with $n_{r}=O(n_{1}^{\alpha})$, $\alpha\in \mathbb{R}$ and
$\alpha\geq1$. For every triple of positive integers  $k,\ell,r$
with $k\geq 3$ and $r\geq 3$, there exists a positive integer
$N=N(k,\ell,r)$ such that
\begin{displaymath}
rx_{k,\ell}(K_{n_{1},n_{2},\cdots,n_{r}}) = \left\{ \begin{array}{ll}
k & \textrm{if $k< r$}\\
k \ or \ k+1 & \textrm{if $k\geq r, \ell\leq\frac{{r\choose2}\lceil\frac{k}{r}\rceil^{2}}{\lfloor\frac{k}{r}\rfloor}$} \\
k+1 & \textrm{if $k\geq r, \ell>\frac{{r\choose2}\lceil\frac{k}{r}\rceil^{2}}{\lfloor\frac{k}{r}\rfloor}$}\ \ \ \
\end{array} \right.
\end{displaymath}
for every integer $n_{1}\geq N$.
Moreover, when $k=r=3$, $\ell\leq\frac{{r\choose2}\lceil\frac{k}{r}\rceil^{2}}{\lfloor\frac{k}{r}\rfloor}=3$,
there exists a positive integer $N=N(\ell)$ such that $rx_{3,\ell}(K_{n_{1},n_{2},n_{3}})=3$ for every integer $n_{1}\geq N$.
\end{theorem}

\section{Concluding remarks}

In this paper, we determine the $(k,\ell)$-rainbow index for some
complete bipartite graphs $K_{m,n}$ with $m=O(n^{\alpha})$,
$\alpha\in \mathbb{R}$ and $\alpha\geq1$. But, we have no idea of
the $(k,\ell)$-rainbow index for general complete bipartite graphs,
e.g., when $m=O(2^{n})$, is $rx_{k,\ell}(K_{m,n})$ equal to $k+1$ or
$k+2$ ? The question seems not easy, since even for the simplest
case $k=2$, $rc_{\ell}(K_{m,n})=3$ or $4$ for sufficiently large
$m,n$ is still an open problem; see \cite{Fujita}.

It is also noteworthy that we use the bipartite Ramsey number in
Lemma \ref{lem2} to show that $rx_{k,\ell}\geq k+1$ for sufficiently
large $n$. But, unfortunately the multipartite Ramsey number does
not always exist; see \cite{David}. Instead, we analyze the
structure of $S$-trees in complete multipartite graphs and give some
lower bounds for $rx(K_{r\times n})$. Since these bounds are weak
(they do not involve coloring), we can not tell
$rx_{k,\ell}(K_{r\times n})=k$ or $k+1$ when $k\geq r$ and
$\ell\leq\frac{{r\choose2}\lceil\frac{k}{r}\rceil^{2}}{\lfloor\frac{k}{r}\rfloor}$,
except for the simple case $k=3$; see Lemma \ref{lem4}. An answer to
this question would be interesting.

\end{document}